\documentclass[12pt, reqno]{amsart}
\usepackage{hyperref}
\usepackage{graphicx,epsfig}
\usepackage[numbers,comma,sort&compress]{natbib}

\usepackage{fancyhdr,latexsym,amsmath,amsfonts,amssymb,amsbsy,amsthm,url}

\usepackage{euscript}
\allowdisplaybreaks
\newtheorem{theorem}{Theorem}[section]
\newtheorem{lemma}{Lemma}[section]
\newtheorem{proposition}{Proposition}[section]

\theoremstyle{definition}
\newtheorem{definition}{Definition}[section]

\theoremstyle{remark}
\newtheorem{remark}{Remark}[section]

\numberwithin{equation}{section}

\usepackage{euscript}
\newlength{\myfboxsep}
\newlength{\mywidth}
\DeclareMathOperator{\E}{E}

\DeclareMathOperator{\Tr}{Tr}

\newcommand{\comment}[1]{}

\begin{document}
\title[Triangular matrices]{Spectral properties of random Triangular matrices}
\author[R. Basu]{Riddhipratim Basu}
\thanks{R. Basu's research is supported by Lo\`eve Fellowship, Department of Statistics, University of California, Berkeley.}
\address{Department of Statistics
 University of California, Berkeley}
\email{riddhipratim@stat.berkeley.edu}
\author[A.Bose]{Arup Bose}
\thanks{ A. Bose's research
supported by J.C.Bose National Fellowship, Dept. of Science and Technology, Govt. of India.}
\address{Statistics and Mathematics Unit Indian
Statistical Institute 203 B.T.~Road Kolkata 700108 India}
\email{bosearu@gmail.com}
\author[S. Ganguly]{Shirshendu Ganguly}
\address{Department of Mathematics, University of Washington,
Seattle}
\email{sganguly@math.washington.edu}
\author[R. S. Hazra]{Rajat Subhra Hazra}
\address{Institut f\"ur Mathematik\\
Universit\"at  Z\"urich\\
Winterthurerstrasse 190\\
CH-8057, Z\"urich
} \email{rajatmaths@gmail.com}

\keywords{Triangular matrices, Wigner, Hankel, Toeplitz and
Symmetric Circulant matrices, limiting spectral distribution,
asymptotically free, $DT$-operators, Catalan words, symmetric words,
semicircular law, joint convergence of matrices.}
\subjclass[2000]{Primary 60B20; Secondary 46L54.}
\date{5th October, 2011. First revision January 26, 2012. Second revision April 06, 2012.}
\begin{abstract}
We prove the existence of the limiting spectral distribution (LSD) of symmetric triangular
patterned matrices and also establish the joint convergence of sequences of such matrices. For the particular case of the symmetric triangular Wigner matrix, we derive expression for the moments of the LSD
using properties of Catalan words. The problem of deriving explicit formulae for the moments  of the LSD
does not seem to be easy to solve for other patterned matrices.
The LSD of the non-symmetric triangular Wigner matrix also does not seem to be easy to establish.
\end{abstract}

\maketitle
\section{\bf Introduction}

%\textbf{Invariant Subspace problem and DT-operators:}

Two of the most important problems in operator theory are the
\textit{invariant subspace problem} and the \textit{hyper invariant
subspace problem}. Let $\mathcal H$ be a separable Hilbert space
(infinite dimensional) and let $\mathcal B(\mathcal H)$ be the
algebra of bounded operators on $\mathcal H$. An {\it invariant
subspace} of $\mathbb A\in\mathcal B(\mathcal H)$ is a subspace
$\mathcal H_0\subset \mathcal H$ such that $\mathbb A(\mathcal
H_0)\subset \mathcal H_0$ and a {\it hyper invariant subspace} of $\mathbb
A$ is a subspace $\mathcal H_0\subset \mathcal H$ that is invariant
for every operator $\mathbb B\in\mathcal B(\mathcal H)$ that
commutes with $\mathbb A$. The invariant subspace problem asks
whether every operator in $\mathcal B(\mathcal H)$ has a closed
non-trivial invariant subspace. The hyper invariant subspace
conjecture states that every operator in $\mathcal B(\mathcal H)$
that is not a scalar multiple of the identity operator has a closed,
non-trivial hyper invariant subspace. For many years there were
attempts to prove this result. Recently there have been
%were
attempts to
disprove this conjecture.
%Initial attempts were made using
%by
%For example, Voiculescu's matrix model for
%circular operator
%%and it had a
%has a large number
%%amount
%of invariant subspaces.

A natural candidate which
stood out for counter examples was the $DT$-operators defined
by~\citet{dykema:haagerup:2004}. This is related to the following
\textit{asymmetric} matrix $T_n$ defined and studied in their work.   It
%This matrix
is a triangular version of the celebrated Wigner matrix
(which  is a symmetric matrix and, in its simplest form, has i.i.d.
entries).

%They the following upper triangular matrix $T_n$ was considered,
\begin{equation}\label{def: T_n}
  T_{n}\  = \left[ \begin{array} {cccccc}
            t_{1,1} & t_{1,2} & t_{1,3} & \ldots & t_{1,n-1} & t_{1,n} \\
            0 & t_{2,2} & t_{2,3} & \ldots & t_{2,n-1} & t_{2,n} \\
            0 & 0 & t_{3,3} & \ldots & t_{3,n-1} & t_{3,n} \\
                  &       &       & \vdots &       &   \\
            0 & 0 & 0 & \ldots & 0 & t_{n,n}
            \end{array} \right]
\end{equation}
 where $(t_{i,j})_{1\leq i\leq j\leq n}$ are i.i.d.\ complex Gaussian random variables having mean $0$ and variance $1/n$.

Let $\mu$ be a compactly supported measure on the complex plane
$\mathbb C$ and let $D_n$ be a diagonal matrix with $\mu$-distributed
independent random variables which are  also independent of $T_n$. Let $\tau_n$ denote the functional
$\frac1n \Tr$. \citet{dykema:haagerup:2004}
%They
showed that $Z_n=c
n^{-1/2}T_n+D_n$ converges in $*$-moments, that is,
\begin{equation}\label{dt-convergence}
\lim_{n\to\infty}\tau_n\left(Z_n^{\epsilon(1)}Z_n^{\epsilon(2)}\ldots
Z_n^{\epsilon(k)}\right)\end{equation} exists for every $k\in\mathbb
N$ and for all $\epsilon(1),\epsilon(2),\ldots,
\epsilon(k)\in\{*,1\}$.
%$DT(\mu,c)$ is an element
%$Z$
%in some
%$*$-non-commutative probability space whose $*$-moments are defined
%by the limit~\eqref{dt-convergence}.
%showed that
Further, the limits $DT(\mu, c)$ in~\eqref{dt-convergence} are
operators and may be viewed as
%$DT$-elements are operators on some Hilbert spaces. In fact, these
%are elements of $(L(\mathbb
%F_2),\tau)$ where $L(\mathbb F_2)\subset \mathcal B(\ell^2(\mathbb
%F_2))$ is the von Neumann algebra generated by the left regular
%representations of the non-abelian free group of two generators and
%$\tau$ is the canonical tracial state.
%$DT(\mu,c)$
%is a $DT(\mu,c)$
elements of a $W^*$ non-commutative probability space
$(\mathcal M,\tau)$ where $\mathcal M$ is a von Neumann algebra and
$\tau$ is a normal state. They also showed that these $DT$-operators are
decomposable (an operator $K$  is decomposable if for every cover
$\mathbb C=U\cup V$ of the complex plane of open subsets $U$ and
$V$ there are $K$-invariant closed subspaces $\mathcal H'$ and
$\mathcal H''$ such that $\sigma(K|_{\mathcal H'})\subset U$ and
$\sigma(K|_{\mathcal H''})\subset V$ and $\mathcal H=\mathcal
H'+\mathcal H''$). As a consequence, it was shown in \citet{dykema:haagerup:2004,dykema:haagerup:2004a}
%Hence it follows that
that $DT$-operators whose
spectra contains more than one point have non-trivial closed hyper
invariant subspaces.
%It was also shown in
 %that
%$DT$-operators whose spectra are singletons have closed, non-trivial
%hyper invariant subspaces. ???? The two previous statements, together, look odd??
Thus the study of the spectrum of the
operator $DT$
%in operator algebras
is of immense importance.

Some
%Few
properties of the $*$-limit of $n^{-1/2}T_n$ were derived by
\citet{dykema:haagerup:2004, dykema:haagerup:2004a},
\citet{sniady:2003} and
%their Stieltjes transform
%for these type of matrices
%was studied by
\citet{shlyakhtenko:1996}. In particular,
considerable analytical techniques as well as
%and also
combinatorics of operator valued free
probability theory were used to show that
%derive the following:
\begin{equation}\label{eq:triangular moments}
{\rm E}\frac{1}{n} {\rm Tr}\left(\frac{T_n^* T_n}{n}\right)^k
\rightarrow \frac{k^k}{(k+1)!}.
\end{equation}
Moreover, the sequence $\{\frac{k^k}{(k+1)!}\}$  are moments of a
probability measure $\nu$ supported on $[0,e]$ and  given by
\begin{equation}
d\nu(x)=\psi(x)dx \quad \text{where } \psi:(0,e)\to \mathbb R^+
\end{equation}
 is the unique solution of
\begin{equation}\label{eq:densityequation}
\psi\left(\frac{\sin v}{v}\exp(v\cot v)\right)=\frac1\pi \sin v \exp(-v\cot v).
\end{equation}
%The Gaussian assumption on the entries was very crucial to these results. ???
The plot of the above function is given in figure~\ref{fig:psi}. 
%The Gaussian assumption on the entries was very crucial to these results. ???
\begin{figure}[htp]
\centering
\includegraphics[height=100mm, width =100mm ]{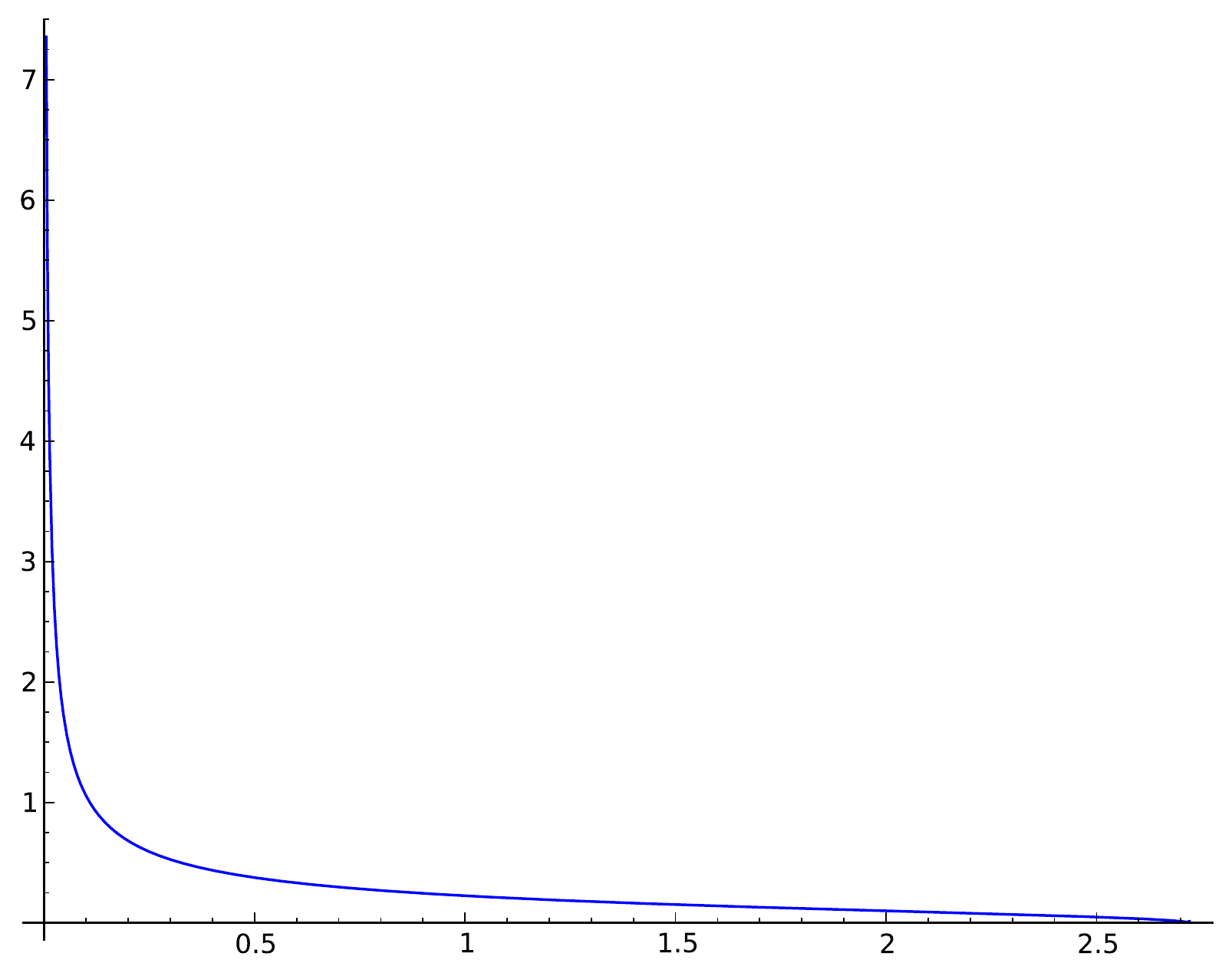}
\caption{\footnotesize Plot of function $\psi$}
%Histogram plot of empirical distribution of triangular  matrix($n=1000$) with $N(0,1)$ entries.}
\label{fig:psi}
%\label{fig:case1&2}
\end{figure}
\vskip20pt
Now consider
%Define
the triangular symmetric
Wigner matrix:
\begin{equation}\label{def:wigner7}
        W_n^u =
        %\frac1{\sqrt n}
            \left[ \begin{array} {cccccc}
                x_{11} & x_{12} & x_{13} & \ldots & x_{1(n-1)} & x_{1n} \\
                x_{12} & x_{22} & x_{23} & \ldots & x_{2(n-1)} & 0 \\
                      &       &       & \vdots &       &   \\
                x_{1(n-1)}&x_{2(n-1)}& 0& \ldots  & 0 & 0 \\
                x_{1n} & 0 & 0 & \ldots & 0 & 0
            \end{array} \right],
            \end{equation}
  where $\{x_{i,j}\}$ are independent and identically distributed random variables.

%The matrix $W_n^u$ has connections with $DT$-operators.
Note that if we
%one
remove the symmetry from the above matrix and consider the matrix
with i.i.d.\ entries above the anti-diagonal,  then it can easily be linked to $T_n$.
%one can easily relate this to the $DT$-operators.
If we multiply this
%e above
matrix by a
matrix $P$ whose  entries on the anti-diagonal are $1$ and the rest
are zero, then we get a matrix of the form $T_n$.
%n upper triangular matrix with %i.i.d. entries.
It can also be seen that $(W_n^u)^2=T_n T_n^*$ where $T_n$
is an upper triangular matrix \eqref{def: T_n} but has dependent
rows and columns now.

Our goal is to study the spectral properties of $n^{-1/2}W_n^u$.
More  generally, we consider symmetric triangular patterned random
matrices and whose entries are not necessarily Gaussian.

We use the classical moment method to show that the limiting
spectral distribution (LSD) of triangular Wigner, Toeplitz,  Hankel
and Symmetric
%AB Reverse
Circulant  matrices exist when the entries have mean
zero and variance one and, are either (i) independent and uniformly
bounded or (ii) i.i.d.

The joint convergence of $p$ sequences of random matrices is defined by the
convergence of $\frac{1}{n}\E \Tr (q)$ for all  monomials $q$ of the $p$ matrices of order $n$.
 The LSD result for one patterned matrix at a time can be extended to the joint convergence
 of any combination of these patterned matrices when we assume that the sequences are independent
 and the entries are independent with mean zero and variance one and
 have uniformly bounded moments of all orders. All the limit results are universal (do not depend on the underlying
distribution of the entries).
% with this triangular structure.

The identification of the LSD
%%AB change in writing
of general patterned matrices  is not a
trivial problem. For
%There are
certain
%%AB change in writing
(full) circulant type matrices this can be done
%where one can find the LSD
(see for example \citet{arup:joydip, massey:miller, kologlu:2011}). One may recall that for the full Toeplitz and Hankel matrices,
no formulae of any sort are known for the LSDs or their  moments (see \citet{hammond:miller, bry, bosesen}).
%%AB. intercnaged the order of the next two sentences,
By  exploiting the Catalan recursions
inbuilt in the Wigner matrix, we show that the $(2k)$th moment of the
LSD for $W_n^u$ is given by (\ref{eq:triangular moments}) and the
odd moments are zero.
However, moment formulae for other matrices
%AB change in writing
%other than the triangular Wigner limit
remain elusive and appear rather difficult to obtain.

Identification of any structure in the joint limit moments appears much harder, even for the
triangular Wigner matrix. Unlike the full Wigner, where independent sequences of Wigner matrices
are asymptotically free, the triangular Wigner matrices are not so. It would be interesting to
unearth the structure of the joint limit for the triangular Wigner matrices. Likewise, the
full Reverse Circulant joint limit satisfies half independence with symmetric Rayleigh marginals
 and the full Symmetric
Circulant joint limit is totally independent with Gaussian marginals (see \citet{bosehazrasaha}).
However, the joint moment structure for  their triangular versions
 seems rather hard to obtain. We do not even know the LSD or the moments of the marginals.

In Section \ref{sec:general} we
%the next section we
describe our setup, notation and state and prove the main existence results for the LSD. In Section \ref{sec:lsdwigner} we concentrate on the triangular Wigner
matrix and derive the moments of the LSD. In Section \ref{sec:comments} we outline how our
results may be extended to joint convergence and also discuss many interesting questions which
arise in the study of random triangular matrices.

\section{\bf Patterned triangular matrices and method of moments}\label{sec:general}
For recent developments on random patterned matrices we refer
the readers
to~\citet{icm}. We adopt the method developed in~\citet{bry}
and~\citet{bosesen}. Here
%AB changed writing
%Below we give
is a quick description of the concepts and notation we
borrow from these works.

A sequence or bi-sequence of variables $\{ x_i; i \geq 0\}$ or $\{
x_{ij};\ i, j\geq 1 \}$ will be called an  \textit{input sequence}.
Let  $\mathbb Z$ be the set of all integers and let $\mathbb
Z_{+}$ denote the set of all non-negative integers. Let
\begin{equation}\label{eq:link}
L_n: \{1, 2, \ldots n\}^2 \to \mathbb Z^d, \  n \geq 1, \ d=1,2
\end{equation} be a sequence of functions such that
$L_{n+1}(i,j)=L_n (i,j)$ whenever $ 1\leq i,  j \leq n$. We shall
write $L_n=L$ and call it the \textit{link} function and by abuse of
notation we write $\mathbb Z_{+}^2$ as the common domain of $\{ L_n
\}$.
%%AB changed writing
%We  define
Patterned matrices are
%to be
matrices of the form
\begin{equation}\label{eq:patternmatrix}
X_n=((x_{L(i,j)})).
\end{equation}
Throughout this article we assume that the
triangular version is derived from a patterned random matrix $X_n$.
%whose link function satisfies the \textit{Property B}:
%\begin{equation*}\Delta(L) = \sup_n \sup_{ t } \sup_{1 \leq k \leq n} \# \{ l:
%1 \leq l \leq n, \ L(k,l) =t\} < \infty .\end{equation*}
Let
$X_n^u$-be the upper triangular version of the  $X_n$ matrix where the
$(i,j)$-th entry of $X_n^u$ is given by $x_{L(i,j)}$ if $(i+j)\leq
n+1$ and $0$ otherwise. We shall often drop the subscript $n$, and
simply denote the matrix by $X^u$. The  triangular Wigner matrix has
already been defined. Here are examples of the triangular versions
of some more common patterned matrices:\\
%\newpage

{ \bf 1. Triangular  Hankel matrix}.
\begin{equation}\label{def:wigner9}
        H_n^u =
            \left[ \begin{array} {cccccc}
                x_{2} & x_{3} & x_{4} & \ldots & x_{n} & x_{n+1} \\
                x_{3} & x_{4} & x_{5} & \ldots & x_{n+1} & 0 \\
                      &       &       & \vdots &       &   \\
                x_{n}&x_{n+1}& 0& \ldots  & 0 & 0 \\
                x_{n+1} & 0 & 0 & \ldots & 0 & 0
            \end{array} \right].
\end{equation}
{ \bf 2. Triangular  Toeplitz matrtix}.
\begin{equation}\label{def: toep}
  T_{n}^u \  = \left[ \begin{array} {cccccc}
            x_{0} & x_{1} & x_{2} & \ldots & x_{n-2} & x_{n-1} \\
            x_{1} & x_{0} & x_{1} & \ldots & x_{n-3} & 0 \\
            x_{2} & x_{1} & x_{0} & \ldots & 0 & 0 \\
                  &       &       & \vdots &       &   \\
            x_{n-1} & 0 & 0 & \ldots & 0 & 0
            \end{array} \right].
\end{equation}
{\bf 3. Triangular  Symmetric Circulant Matrix}.
\begin{equation}\label{def: symm_circ}
    S_{n}^u \  = \left[ \begin{array} {cccccc}
            x_{0} & x_{1} & x_{2} & \ldots & x_{2} & x_{1} \\
            x_{1} & x_{0} & x_{1} & \ldots & x_{3} & 0 \\
            x_{2} & x_{1} & x_{0} & \ldots & 0 & 0 \\
                  &       &       & \vdots &       &   \\
            x_{1} & 0 & 0 & \ldots & 0 & 0
            \end{array} \right].
\end{equation}
It is to be noted that a triangular  Reverse Circulant matrix is the
same as a triangular  Hankel matrix.

To prove the existence of LSD via moment method, it suffices to
verify the following three conditions.
\begin{enumerate}
\item {For every $k\geq 1$, $\frac{1}{n}{\rm E}\left[{\rm Tr}\left(\frac{1}{\sqrt{n}}X_n^u\right)^k\right] \rightarrow \beta_k$
as $n \rightarrow \infty$ (Condition M1).}

\item For every $k\geq 1$, \begin{equation*}{\rm E}\left[\frac1n{\rm Tr}\left((\frac{1}{\sqrt{n}}X_n^u)^k)-\E(\frac1n\Tr((\frac{1}{\sqrt{n}}X_n^u)^k)\right)\right]^4=O\left(n^{-2}\right)\text{ as } n \rightarrow \infty \quad (\text{Condition M4}).
\end{equation*}
\item {$\sum_k \beta_{2k}^{-1/2k}= \infty$ (Carleman's Condition).}
\end{enumerate}

The two main conditions we need are: \\

\noindent
{\bf Assumption A:}
 $\{ x_i\}$ or $\{ x_{ij}\}$ are independent with mean 0 and
%%AB changed language
variance 1 and \textit{either} (i) have uniformly bounded moments of all orders  \textit{or}
(ii) are identically distributed. \\

The restriction of the so called Property B of \citet{bosesen} to the triangular
matrices is given below.\\

\noindent
{\bf Assumption B:} The link function $L$ satisfies
$Property~B'$, i.e.,
\begin{equation*}\Delta = \sup_n \sup_{ t } \sup_{1 \leq k \leq n} \# \{
l: 1 \leq l \leq n, k+l\leq n+1, \ L(k,l) =t\} < \infty
.\end{equation*}
Since the moment method entails computing all moments, we shall also use the
following assumption. \\
%By suitable approximation
% results as in \citet{bosesen}, we shall be able to relax thcan
%relax this to the following: \\

\noindent
{\bf Assumption C:}  $\{x_i\}$ or $\{x_{ij}\}$ are independent
with mean $0$, variance 1, and are  uniformly bounded.\\

In the course of verifying  conditions (1), (2) and (3), we shall need the notion of circuits, words, etcetera from  \citet{bosesen}.

Any function\  $\pi:\{0,1,2,\cdots,h\} \rightarrow
\{1,2,\cdots,n\} $ \ with  $\pi(0) = \pi(h)$ is called a
\textit{circuit} of
%%AB ?? is the notation l(\pi) used anywhere??? if not then drop it. jusst say length h...
\textit{length} $h$.    The dependence
of a circuit on  $h$ and $n$ will be suppressed.
%\noindent \textit{Equivalence of circuits:}

Two circuits $\pi_1$ and $\pi_2$
are \textit{equivalent} if and only if their $L$ values
respectively match at the same locations. That is,
%{\fontsize{9.5}{12}
$$\big\{L(\pi_1(i-1), \pi_1(i)) = L(\pi_1(j-1),
\pi_1(j)) \Leftrightarrow L(\pi_2(i-1), \pi_2(i)) = L(\pi_2(j-1),
\pi(j))\big\}.$$

\noindent Any equivalence class
can be indexed by a partition of $\{1,2,\cdots,h\}$. Each block of
a given partition identifies the positions where the $L$-matches
take place. We can label these partitions by \textit{words} of letters where  the first occurrence of each letter  is in alphabetical order. For example if $h = 5$ then the partition
$\{ \{1,3,5\}, \{2,4\}\} $ is represented by the word $ababa$.
This identifies the circuits $\pi$ such that $L(\pi(0), \pi(1))=
L(\pi(2), \pi(3))=L(\pi(4), \pi(5))$ and $L(\pi(1),
\pi(2))=L(\pi(3), \pi(4))$. Let $w[i]$ denote the $i$-th
entry of $w$.
%%AB added defn of pair-matched word and defn of |w|.
A word is said to be \textit{pair-matched} if each letter appears exactly twice.
Denote by $|w|$ the length of the word $w$.

%%AB changed wording
For a  pair-matched word $w$ of length $2k$,  let
%us denote
\begin{eqnarray*}
\Pi_{X}^*(w)&=&\{\pi: w[i]=w[j] \Rightarrow
L(\pi(i-1),\pi(i))=L(\pi(j-1),\pi(j))\} \\
\Pi_{X}(w)&=&\{\pi: w[i]=w[j] \Leftrightarrow
L(\pi(i-1),\pi(i))=L(\pi(j-1),\pi(j))\}.
%\Pi_{1,X}^*(w)&=&\Pi_{X}^*(w)\bigcap \Pi_{2k}^u.
\end{eqnarray*}
Note that $\Pi_{X}(w)$ is nothing but the equivalence class of circuits corresponding to the word $w$.
 We shall write now
\begin{eqnarray*}
\frac{1}{n}{\rm Tr}((\frac{1}{\sqrt{n}}X^u)^k) &=&
\frac{1}{n^{1+k/2}}\sum_{\pi:\pi~{\rm circuit}}\ \ \prod_{i=1}^k
x_{L(\pi(i-1),\pi(i))}1_{\{\pi(i-1)+\pi(i)\leq n+1\}}\\ &=&
\frac{1}{n^{1+k/2}}\sum_w \ \ \sum_{\pi\in \Pi(w)}\ \
\prod_{i=1}^k x_{L(\pi(i-1),\pi(i))}1_{\{\pi(i-1)+\pi(i)\leq n+1\}}.
\end{eqnarray*}
Let us denote
\begin{equation*}X_{\pi}*= \prod_{i=1}^k x_{L(\pi(i-1),\pi(i))}1_{\{\pi(i-1)+\pi(i)\leq n+1\}}.\end{equation*}
For a fixed $k$, let us define the class $\Pi_{k}^u$ as follows.
\begin{equation*}\Pi_{k}^u=\{\pi:\pi~{\rm is~circuit~of~length}~k,\  \pi(i-1)+\pi(i)\leq n+1, 1\leq i\leq k\}. \end{equation*}
Let us denote \begin{equation*}\Pi_{1,X}(w)=\Pi_X(w)\bigcap
\Pi_{2k}^u \ \text{ and} \ \Pi_{1,X}^*(w)=\Pi_X^*(w) \bigcap
\Pi_{2k}^u.\end{equation*} For every pair-matched word $w$ of length
$2k$, define, if the limit exists,
\begin{equation*}p_{u,X}(w)= \lim_{n \rightarrow \infty} \frac{\#\Pi_{1,X}^*(w)}{n^{1+k}} \end{equation*}
where for any set $A$, $\#A$ denotes the number of elements of $A$.

When there is no chance for confusion, we simply write $p_u(w)$ for $p_{u, X}(w)$ and $\Pi_1(w)$ and $\Pi_1^*(w)$ for $\Pi_{1,X}(w)$ and $\Pi_{1,X}^*(w)$ respectively.

By using
the arguments of \citet{bosesen}, we immediately get the following theorem.
\begin{theorem}
\label{theorem:convergence}
Let $\{X_n^u\}$ be a sequence of patterned triangular  matrices
satisfying Assumptions A and B. Suppose for every $k\geq 1$ and for
every pair-matched word $w$ of length $2k$, $p_u(w)$ exists. Then the
LSD of $\frac{X_n^u}{\sqrt{n}}$ exists a.s. The LSD is universal, symmetric
about 0 and is determined by the even moments
\end{theorem}
\begin{equation}\label{2kmoment}
\beta_{2k}=\displaystyle\sum_{\substack{{w\ \text{pair-matched,}} \\
|w|=k}}p_u(w), \ \ k\ge1.
\end{equation}
\begin{proof} We provide
a brief sketch of the proof. By using a truncation argument (see for example \citet{bosesen}),
we may work under the stronger Assumption C.
The first lemma in particular shows
that the odd moments are zero and the LSD is symmetric.
\begin{lemma} Suppose Assumptions B and  C hold.
Let $w$ be a matched word of length $k$ which is not pair-matched.
Then
\begin{equation*}\lim_{n\rightarrow \infty}\frac{1}{n^{1+k/2}}\sum_{\pi \in
\Pi_1^*(w)}{\rm E}[X_{\pi^*}]=0.\end{equation*}
\end{lemma}
\begin{proof} First note that if $w$ is non-matched then $\E[X_{\pi^*}]=0$ since the entries have mean zero.
Now consider a $w$ which has more than two matches.
Let $X_{\pi}=\prod_{i=1}^k x_{L(\pi(i-1),\pi(i))}$. Clearly then,
$|X_{\pi^*}| \leq |X_{\pi}|$ and hence ${\rm E}|X_{\pi^*}| \leq {\rm
E}|X_{\pi}|$. It is proved in  \citet{bosesen} that
\begin{equation*}\lim_{n\rightarrow \infty}\frac{1}{n^{1+k/2}}\sum_{\pi \in \Pi^*(w)}{\rm E}|X_{\pi}|=0. \end{equation*}
The lemma follows immediately from this since $\Pi_1^*(w) \subset
\Pi^*(w)$.
\end{proof}
The next lemma helps to verify Condition M4. The proof is a
direct generalization of Lemma 2 of~\citet{bosesen}. We skip
the details.
\begin{lemma}
Let $\{ X_n^u\}$ be a sequence of random patterned triangular  matrices
satisfying Assumptions B and C. Then
\begin{equation*}{\rm E}\left [\frac{1}{n}{\rm Tr}(\frac{1}{\sqrt{n}}X_n^u)^k-{\rm
E}(\frac{1}{n}{\rm
Tr}(\frac{1}{\sqrt{n}}X_n^u)^k)\right]^4=O\left(n^{-2}\right).
\end{equation*}
\end{lemma}
Since there does not exist any pair-matched word of length $k$, for
$k$ odd, it follows from Lemma $2.1$ that
\begin{equation*}\beta_k=\lim_{n\rightarrow \infty}{\rm E}(\frac{1}{n}{\rm Tr}(\frac{X_n^u}{\sqrt{n}})^k)=0, \quad\text{if $k$ is odd}. \end{equation*}
%\end{proof}
It also follows that
\begin{equation*}\beta_{2k}=\lim_{n\rightarrow \infty}{\rm E}(\frac{1}{n}{\rm Tr}(\frac{X_n^u}{\sqrt{n}})^{2k})=\displaystyle\sum_{\substack{{w\  \text{pair-matched,}} \\
|w|=k}}p_u(w).\end{equation*} Hence Condition M1 also holds. From
Lemma 2.2 it follows that Condition M4 is satisfied. It remains to
show that the sequence $\{\beta_{k}\}$ satisfies Carleman's
condition. It follows easily from
%%AB changed to Assumption B
%Property~B
Assumption B that for every
pair-matched word of length $2k$,
\begin{equation*}\#\Pi_1^*(w)\leq \#\Pi^*(w)\leq n^{k+1}\Delta ^k .\end{equation*}
Hence $p_u(w)\leq \Delta^k$ and so $\beta_{2k} \leq \frac{(2k!)}{2^k
k!}\Delta ^k$ and Carleman's condition is thus easily verified.
\end{proof}
%\section{Existence of LSD for Common Triangular  Matrices}
We can use Theorem~\ref{theorem:convergence} to prove the existence of LSD for the
triangular  matrices listed earlier.
\begin{theorem}
Let $\{X_n^u\}$ be any of the following triangular  matrices with input
sequence
%which is independent with mean zero and variance 1 and such that
satisfying Assumption A: triangular  Wigner, triangular  Hankel, triangular
Toeplitz and triangular  Symmetric Circulant. Then the LSD for
$\frac{X_n^u}{\sqrt{n}}$ exists almost surely. The LSDs are universal and are
symmetric with even moments given by (\ref{2kmoment}).
\end{theorem}
\begin{proof} As before, without loss, we shall work under Assumption C.
Notice that the link functions of all these matrices satisfy
%%AB changed to Assumption B
%Property~B.
Assumption B. Thus Carleman's condition and Condition M4 follow
immediately if we can establish Condition M1. The combinatorics to show
the existence of $p_u(w)$ is done case-by-case. However, all the
proofs are similar and use the volume method. We provide a sketch
only for the triangular  Wigner and the triangular  Hankel matrices. In the full
matrix version of the matrices considered here, $p(w)$ is
evaluated as an integral of an indicator function (see for example~\citet{icm}). Here also,
$p_u(w)$ turns out to be an integral, but of a different indicator
function corresponding to $\Pi_1^*(w)$ instead of $\Pi^*(w)$.\\

\noindent{\bf Triangular  Wigner Matrix:} Let $L_W$ denote the Wigner
link function. Following \citet{bosesen}, a word is said to be \textit{Catalan} if it is pair-matched and
 deleting all double letters leads to the empty word. For example $abba$ is Catalan and
 %%AB suppressed: the   word
 $abab$ is not. Now it is shown in  \citet{bosesen} that
if $w$ is not Catalan then,
\begin{equation*}p(w)= \lim_{n \rightarrow \infty} \frac{\#\Pi^*(w)}{n^{1+k}}=0. \end{equation*}
As $\Pi^*(w)\supseteq \Pi_1^*(w)$, it follows that $p_u(w)=0$ if $w$
is non-Catalan. Thus we now focus on Catalan words of length $2k$.
%\end{equation*}\Pi_1^*(w)= \{\pi: w[i]=w[j] \Rightarrow L_W(\pi (i-1),\pi(i))=L_W(\pi (j-1),\pi(j)), \pi(i-1)+\pi(i)\leq n \forall i\}.\end{equation*}
%%AB ????generating vertex not defined???? SO I hav added.
We shall call any $\pi(i)$ (or by abuse of notation any $i$) a \textit{vertex}. Any $\pi(i)$ is  a \textit{generating  vertex} if a letter appears for the first time at $i$ (when read left to right).
We also define $\pi(0)$  (or 0) to be a generating vertex.
\begin{lemma}
\label{catalan3} Let $w$ be a Catalan word of length $2k$. Let $S$
denote the set of all
generating vertices of $w$.
%%AB dropped,  not needed  (including 0).
Then for all $j \notin S$, there exists a unique $i \in S$ such
that $i<j$ and $\pi(j)=\pi(i)$ for all $\pi \in \Pi^*(w)$.
\end{lemma}
\begin{proof}
Let $\pi \in \Pi^*(w)$. Let $j$ be the minimum index of a
non-generating vertex of $w$. Clearly then, $w[j-1]=w[j]$ and hence
$\pi(j-2)=\pi(j)$. Since $j>j-2 \in S$, the result is true in this
case. Now let $j$ be any non-generating vertex. Let us assume that
for every non-generating vertex with index less than $j$, the result
holds. Let $i<j$ be the index of first occurrence of $w[j]$. Let
$w_1$ be the subword formed by letters between $w[i]$ and $w[j]$.
Since $w$ is Catalan, $w_1$ is also Catalan, and it can be easily
shown that, $\pi (i)=\pi(j-1)$ and hence $\pi(j)=\pi(i-1)$. If
$(i-1)\in S$, then we are already done. If $(i-1) \notin S$, then
also by induction hypothesis the result holds.

%The uniqueness follows from the fact that
If the $i$ corresponding to a fixed $j$ is not unique then we have a
non-trivial relation between two generating vertices which implies
$\#\Pi^*(w)=O(n^k)$, and hence contradicting the fact that
$\lim_{n\rightarrow \infty}\frac{\#\Pi^*(w)}{n^{1+k}}=1$. This
proves the uniqueness.
\end{proof}
\begin{definition}\label{def:phi} For any $j$ (not necessarily in $S$), let us denote by $\phi(j)$
the {\it unique} vertex such that
\begin{equation*}\phi(j) \in S, \ \phi(j)\leq j\ \
{\rm  and}\ \ \pi(j)=\pi(\phi(j)) \ \ {\rm for~all} \ \ \pi \in
\Pi^*(w).\end{equation*}
\end{definition}
  Note that if $j
\in S$, then $j=\phi(j)$. Next we note that among the $2k$
equations, $\pi(i-1)+\pi(i)\leq n+1, \ 1\leq i \leq 2k$, each equation
is repeated twice, as $w[i]=w[j]\Rightarrow
\pi(i-1)+\pi(i)=\pi(j-1)+\pi(j)$. So we can write
\begin{eqnarray*}
\Pi_1^*(w)&=& \{\pi: w[i]=w[j] \Rightarrow L_W(\pi (i-1),\pi(i))=L_W(\pi (j-1),\pi(j)),\\
&& \ \ \ \ \ \ \pi(i-1)+\pi(i)\leq n+1 \ \forall i\in S-\{0\}\}\\
&=& \{\pi: \pi(j)=\pi(\phi(j)) \ \ \forall j \notin S,
\pi(\phi(i-1))+\pi(\phi(i))\leq n+1 \ \ \forall i\in S-\{0\} \}.
\end{eqnarray*}
Now we use the standard volume method arguments. Let us define
\begin{equation} \label{def: xx1}
v_i=\frac{\pi(i)}{n}, \  \ U_n=\left
\{\frac{1}{n},...,\frac{n-1}{n}, 1\right \}  \ \ \text{and} \ \
v_S=\{v_i:i\in S\}.
\end{equation}
Then,
%%AB indented the seocnd line more
\begin{eqnarray*}
\#\Pi_1^*(w)&=&\# \{(v_0,...,v_{2k}): v_i\in U_n \ \ \forall 0\leq
i\leq 2k, v_i=v_{\phi(i)} \ \ \forall i \notin S,\\ && \ \ \ \ \ \ \ \  \ \ \ \ \ \ \ \ \ \ \ \ \
v_{\phi(i-1)}+v_{\phi(i)}\leq 1+1/n \ \ \forall i
\in S-\{0\},v_0=v_{2k}\}\\
&=& \# \{v_S: v_i\in U_n \ \ \forall i \in S,
v_{\phi(i-1)}+v_{\phi(i)}\leq 1+1/n \ \ \forall i \in S-\{0\}\}.
\end{eqnarray*}
From the above equation it follows that
$\frac{\#\Pi_{1}^*(w)}{n^{1+k}}$ is nothing but the Riemann sum for
the function
%%AB introduced new notation I_W
$$I_W(v_S)=I( v_{\phi(i-1)}+v_{\phi(i)}\leq 1, i\in S-\{0\})$$
over $[0,1]^{k+1}$. Since the function is clearly Riemann
integrable, the Riemann sum converges to the integral
%%AB used new notation to rewrite next display
\begin{equation} \label{new1}
\lim_{n\rightarrow \infty}\frac{1}{n^{1+k}}\#\Pi_{1}^*(w)=
\int_{[0,1]^{k+1}} I_W(v_S)dv_S.
\end{equation}
%\begin{equation} \label{new1}
%\lim_{n\rightarrow \infty}\frac{1}{n^{1+k}}\#\Pi_{1}^*(w)=
%%\idotsint\limits
%\int_{[0,1]^{k+1}} I( v_{\phi(i-1)}+v_{\phi(i)}\leq 1, i\in
%S-\{0\})dv_S.
%\end{equation}
It follows that $p_u(w)=\lim_{n\rightarrow
\infty}\frac{1}{n^{1+k}}\#\Pi_{1}^*(w)$ exists.\\

\noindent {\bf Triangular  Hankel Matrix:} The Hankel  link function is
$L(i,j)=i+j$. Here
\begin{equation*}
%%AB dotn need a \#  here. I dropped it.
\Pi_1^*(w)=\{\pi: w[i]=w[j]\Rightarrow
\pi(i-1)+\pi(i)=\pi(j-1)+\pi(j), \pi(i-1)+\pi(i) \leq
n+1\}.\end{equation*}
Let $S$ denote the set of all generating vertices of $w$. For
every $i \in S-\{0\}$, let $j_i$ denote the index such that
$w[i]=w[j_i].$ Let us define $v_i, U_n, v_S$ as in (\ref{def:
xx1}). Then
%%AB indented more the seocnd line of display
\begin{eqnarray*}
\#\Pi_1^*(w) &=&\# \{(v_0,...,v_{2k}): v_i\in U_n, \forall~ 0\leq
i\leq 2k, v_{(i-1)}+v_i=v_{(j_i-1)}+v_{j_i}\forall i\in S,\\ & &
 \ \ \ \ \ \ \ \ \ \ \ \ \ \ \ \ \ \ \ \ \   v_{(i-1)}+v_{i}\leq 1+1/n, \forall i\in S-\{0\},v_0=v_{2k}\}.
\end{eqnarray*}
It can easily be seen from the above equations (other than
$v_0=v_{2k}$) that each of the $\{v_i: i\notin S\}$ can be written
uniquely as an integer linear combination $L_i(v_S)$. Moreover,
$L_i(v_S)$ only contains generating vertices of index less than $i$
with non-zero coefficients. For all $i\in S$, let us define
$L_i(v_s)=v_i.$

Clearly,
\begin{eqnarray*}
\#\Pi_{1}^*(w)&=&\# \left \{(v_0,...,v_{2k}): v_i\in U_n ~\forall~
0\leq
i\leq 2k,v_0=v_{2k}, v_i=L_i(v_S)  \forall i\notin S, \right.\\
&&\ \ \ \ \ \ \ \ \ \ \  \ \ \ \ \ \ \ \ \ \  \left. v_{i-1}+v_{i}\leq 1+1/n \ \forall i\in S-\{0\} \right \}.
\end{eqnarray*}
 Integer linear combinations of elements of $U_n$ are again in $U_n$ iff they are  between 0 and 1. Hence,
\begin{equation}\label{riemann7}
\begin{array}{lll}
\#\Pi_{1}^*(w) &=&\# \left \{v_S: v_i\in U_n \ \ \forall i \in S, v_0=L_{2k}(v_S), 0< L_i^l(v_S)\leq 1 \ \ \forall i\notin
S,\right. \\
&&\ \ \ \ \ \ \ \ \ \ \left. L_{i-1}(v_S)+L_i(v_s)\leq 1+1/n \ \ \forall i\in S-\{0\} \right \}.
\end{array}
\end{equation}
From (\ref{riemann7}) it follows that
$\frac{\#\Pi_{1}^*(w)}{n^{1+k}}$ is nothing but the Riemann sum for
the function
%%AB introduced new notation
$$I_H (v_S)=I(0\leq L_i^l(v_S)< 1, i\notin S,
v_0=L_{2k}^l(v_S),L_{i-1}(v_S)+L_i(v_s)\leq 1\ \ \forall i \in
S-\{0\} )$$ over $[0,1]^{k+1}$. Since the function is clearly Riemann
integrable, the Riemann sum converges to the integral
%%AB used new notation to write the next dispaly
\begin{eqnarray*}
\lim_{n\rightarrow \infty}\frac{1}{n^{1+k}}\#\Pi_{1}^*(w)=
\int_{[0,1]^{k+1}}  I_H(v_S)dv_S.
\end{eqnarray*}
%\begin{eqnarray*}
%\lim_{n\rightarrow \infty}\frac{1}{n^{1+k}}\#\Pi_{1}^*(w)&=&
%%\idotsint\limits_
%\int_{[0,1]^{k+1}} I(0\leq L_i^l(v_S)< 1, i\notin S,
%v_0=L_{2k}^l(v_S),\\ \nonumber && \ \ \ \ \
%L_{i-1}(v_S)+L_i(v_s)\leq 1 \ \ \forall i \in S-\{0\})dv_S.
%\end{eqnarray*}
Hence $p_u(w)$ exists for every pair-matched word $w$.

Incidentally,  following \citet{bosesen}, a pair-matched word is said to be
\textit{symmetric} if every letter appears once in an even and once in an odd position.
For a Hankel matrix,  $p(w)=0$ if $w$ is not a symmetric word. Hence it also follows that $p_u(w)=0$ for every such
%%AB changed writing a bit
%asymmetric
word $w$ for the triangular Hankel matrix.
%\noindent {\bf Triangular  Toeplitz and triangular  Symmetric Circulant
%Matrices:} It can be shown along the lines similar to the previous
%two cases $p_u(w)=\lim_{n\rightarrow
%\infty}\frac{|\Pi_1^*(w)|}{n^{1+k}}$ exists. Hence the proof is
%complete.
\end{proof}
%%AB cut pasted and modified Risshi's proof write up
The LSD of the Reverse Circulant, the Symmetric Circulant, the Toeplitz and
the Hankel matrices are unbounded.  The same is true for their triangular versions.
\begin{theorem}
The LSD of triangular Hankel,  Toeplitz and Symmetric Circulant matrices have unbounded support.
\end{theorem}
\begin{proof}
We shall prove  the theorem only for the triangular Hankel matrices. The proof for other matrices
is similar and is omitted. Recall that the $2k$-th moment of the LSD
for the full and triangular versions are given by
$$\beta_{2k}=
\sum_{\substack{{w\ \text{pair-matched,}} \\
|w|=k}}
p_u(w), \ \
\beta^\prime_{2k}=
\sum_{\substack{{w\ \text{pair-matched,}} \\
|w|=k}}
p(w).$$
%\begin{lemma}
%\label{hankelsupportlemma}
%Let $w$ be a pairmatched word of length $2k$.  Let $p_{u}(w)$ denote the contribution of the word $w$ in the $2k$-th
%moment of the triangular Hankel matrix.  Let $p(w)$ denote the corresponding contribution of the word $w$ in the $2k$-th moment of the (full) Hankel matrix.

We claim that for every $w$,
$$p_u(w)\geq \frac{1}{2^{k}}p(w).$$
To see this,
%\begin{proof}
first note that  using the already introduced notation,
% from above
%%AB rewritten next dispplaywith new notation
\begin{equation}
\label{hankelsupport1}
p_u(w)=\int_{[0,1]^{k+1}} I_H(v_S)dv_S.
\end{equation}
%\begin{equation}
%\label{hankelsupport1}
%p_u(w)=\int_{[0,1]^{k+1}} I(0\leq L_i^l(v_S)< 1, i\notin S,
%v_0=L_{2k}^l(v_S), L_{i-1}(v_S)+L_i(v_s)\leq 1 \ \ \forall i \in S-\{0\})dv_S.
%\end{equation}
%%AB added a few words
Also, clearly, by a similar logic for the full Hankel matrix,
$$p(w)=\int_{[0,1]^{k+1}} I(0\leq L_i^l(v_S)< 1, i\notin S,
v_0=L_{2k}^l(v_S))dv_S. $$
Making a change of variable $y_i=v_i/2$ for all $i\in S$, since $L_i^l$ are linear functions, we see that
%%AB itnrdocued new notation and rewrote display
\begin{eqnarray}
\label{hankelsupport3}
\frac{p(w)}{2^k}&=&\int_{[0,1/2]^{k+1}} I(0\leq L_i^l(y_S)< 1/2, i\notin S,
y_0=L_{2k}^l(y_S))dv_S \nonumber \\ &=& \int_{[0,1/2]^{k+1}} K(v_S) dv_S \ \ \text{say}.
\end{eqnarray}
Now it is trivial to note that
%%AB rewrote next display with new notation.
$$K(v_S) \leq I_H(v_S) \ \ \text{for all} \ \ v_S.$$
%\begin{eqnarray*}
%&~&\{0\leq L_i^l(y_S)< 1/2, i\notin S,
%y_0=L_{2k}^l(y_S)\}\cap [0,1/2]^{k+1}\\
%&\subseteq & \{0\leq L_i^l(y_S)< 1, i\notin S,
%v_0=L_{2k}^l(y_S), L_{i-1}(y_S)+L_i(y_s)\leq 1~ \forall i \in S-\{0\}\}\cap [0,1]^{k+1}
%\end{eqnarray*}
The claim now follows from (\ref{hankelsupport1}) and (\ref{hankelsupport3}).
%\end{proof}
%Using Lemma \ref{hankelsupportlemma} we obtain that
As a consequence, $$\beta_{2k}\geq \frac{\beta^\prime_{2k}}{2^k}$$ and the result now follows from noting that the LSD of Hankel matrix has unbounded support.
\end{proof}
\section{\bf The LSD of Triangular  Wigner Matrices}\label{sec:lsdwigner}

As mentioned earlier, in general the LSD is hard to identify and we
have the answer only for the triangular Wigner matrix.
For the full Wigner matrix, $p(w)=1$ for each Catalan word, and as a consequence the LSD
is the semicircle law.
 In the
following table we list the values of $p_u(w)$ for Catalan
words of small lengths. It may be observed that the contributions
are equal within certain isomorphic classes but they are unequal in general. Though it does not
seem easy to obtain the individual $p_u(w)$'s for different
Catalan words, it is, however possible to calculate their total
contribution in a relatively simple way.
We turn to this direction now.\\

\begin{table}[h]
\begin{center}
\caption{$p_u(w)$ for Catalan words for $W_n^u$ }
\begin{tabular}{|c|c|}\hline
 Word & $p_u(w)$ \\
\hline
 aa & 1/2 \\
\hline
 aabb & 1/3 \\
\hline
 abba & 1/3 \\\hline
 aabbcc & 1/4\\\hline
 abbcca & 1/4 \\
\hline
 abbacc & 5/24\\\hline
 aabccb & 5/24 \\\hline
 abccba & 5/24 \\
\hline
\end{tabular}
\end{center}
\end{table}
\begin{theorem}\label{theorem:moments}
Let $W_n^u$ be a triangular  Wigner matrix with an input sequence
satisfying Assumption C. Then almost surely
\begin{equation*}\lim_{n\rightarrow \infty}\frac{1}{n}{\rm Tr}\left(\frac{W_n^u}{\sqrt{n}}\right)^{2k}=\frac{k^k}{(k+1)!}. \end{equation*}
Hence the density of the LSD is given by $|x|\psi(x^2)$ where $\psi$ satisfies \eqref{eq:densityequation} and its support is contained in $[-\sqrt e,\ \sqrt e]$.
\end{theorem}
Note that since the LSD is universal, the above limit moments are also the moments  of the LSD
under Assumption A.
\begin{remark}
 %%AB changed wording
 %It is interesting to note that
 Interestingly, the above moment sequence is directly related to the Lambert $W$ function, which is defined through the following functional equation
\begin{equation*}
 W(z)\exp(W(z))=z.
\end{equation*}
If $x$ is real then the above equation has two possible branches and the one satisfying $W(x)\geq -1$ is called the
principal branch and is denoted by $W_0$. The principal branch $W_0$ is analytic at zero and one can use Lagrange's inversion
theorem to see that,
$$1+\frac{1}{xW_0(1/x)}= \sum_{k=0}^{\infty}\frac{k^k}{(k+1)!}x^{-(k+1)}$$ for $x$ lying outside $(-e,e)$.
For details we refer the reader to the exciting article on Lambert $W$ function by \citet{corless:1996}.
\end{remark}
%\noindent{\bf Proof}
\begin{proof}
To begin with, we need the following the simple lemma whose
proof is omitted.
\begin{lemma}
\label{circuit1} (i) Let $w$ be a Catalan word of length $2k$. Let
$\pi$ be a  function $\pi: \{0,1,...,k\}\rightarrow \{1,2,...,n\}$
such that $w[i]=w[j] \Rightarrow
(\pi(i-1),\pi(i))=(\pi(j),\pi(j-1))$. Then $\pi(0)=\pi(2k)$ and
hence $\pi \in \Pi^*(w)$.

(ii) Let $\phi(j)$ be as in Definition~\ref{def:phi}. If $w=w_1w_2$ where $w_1$ and $w_2$ are Catalan words of length
$2k_1$ and $2k_2$ respectively. Then for every vertex $i>2k_1$,
%%AB why split up? I combined > and = into geq
$\phi(i)\geq 2k_1$.
% or $\phi(i)=\phi(2k_1).$
\end{lemma}
Now let $w$ be a Catalan word of
length $2k$. Let $S$ be the set of generating vertices for $w$. It
has already been shown that (see (\ref{new1}))
\begin{eqnarray} \label{new2}
p_u(w)=\idotsint\limits_{[0,1]^{k+1}} I(
v_{\phi(i-1)}+v_{\phi(i)}\leq 1, i\in S-\{0\})dv_S.
\end{eqnarray}
This integral can be evaluated as an iterated integral. It is clear
that integrating out all the variables other that $v_0$ leaves a
polynomial in $v_0$, say $Q_w$. It follows that
\begin{equation}
p_u(w)=\int_0^1 Q_w(v_0)dv_0
\end{equation}
for some \textit{polynomial} $ Q_w(\cdot)$.  Let us illustrate with
a few examples.
 Let $w=aa$. Then
 \begin{equation*}p_u(w)=\int_{v_0+v_1\leq 1}dv_0dv_1=\int_0^1(1-v_0)dv_0\end{equation*} and hence
 \begin{equation*}Q_{aa}(x)=1-x.\end{equation*}
Now let $w=abba$. Then
   \begin{eqnarray*}
   p_u(w)&=&\idotsint\limits_{[0,1]^3}
I(v_0+v_1\leq 1, v_1+v_2 \leq 1)dv_0dv_1dv_2\\
&=&\int_0^1\int_0^{1-v_0}\int_0^{1-v_1}dv_2dv_1dv_0\\
&=&\int_0^1\int_0^{1-v_0}(1-v_1)dv_1dv_0
=\int_0^1\frac{(1-v_0^2)}{2}dv_0.
   \end{eqnarray*}
Hence
\begin{equation*}Q_{abba}(x)=\frac{1-x^2}{2}.\end{equation*}
%\end{enumerate}
%We begin with a study of simple properties of these polynomials $p_w$.
The following two lemmas collect the required properties of $Q_w(\cdot)$.
\begin{lemma}
(i) Let $w=w_1w_2$ be a Catalan word of length $2k$ where $w_1$ and $w_2$ are both Catalan. Then $Q_w(x)=Q_{w_1}(x)Q_{w_2}(x)$.\\

\noindent (ii)
Let $w=aw_1a$ be a Catalan word with $w_1$ Catalan. Then
\begin{equation*}Q_w(x)=\int_0^{1-x}Q_{w_1}(y)dy.\end{equation*}
%\end{lemma}
\end{lemma}
\begin{proof}
(i) Let $w_1$ and $w_2$ be Catalan words of length $2k_1$ and $2k_2$
respectively.
We divide the set of inequalities in the indicator function in
(\ref{new2}) above into two classes:
\begin{enumerate}
\item
$v_{\phi(i-1)}+v_{\phi(i)}\leq 1, i\in S-\{0\}, i\leq 2k_1$; i.e., the inequalities corresponding to generating vertices in $w_1$.

\item
$v_{\phi(i-1)}+v_{\phi(i)}\leq 1, i\in S-\{0\}, i> 2k_1$; i.e., the inequalities corresponding to generating vertices in $w_2$.
\end{enumerate}
From Lemma \ref{circuit1} (i) it follows that $\phi(2k_1)=0$ and
from Lemma \ref{circuit1} (ii) it follows that the inequalities  in items (1)
and (2) do not have any common variable except $v_0$. Hence in
(\ref{new2}) integrating w.r.t the variables corresponding to the
generating vertices in $w_1$ we get $Q_{w_1}(v_0)$ and integrating
w.r.t. the variables corresponding to the generating vertices in
$w_2$, we get $Q_{w_2}(v_0).$ The result now follows from the
definition of $Q_w.$\\
%\end{proof}

%\begin{lemma}
%\begin{proof}
\noindent (ii) The proof of this  is similar  and we only give a sketch. Note that
the variable $v_0$ does not occur anywhere in the equations
corresponding to the generating vertices in $w_1$. Integrating
w.r.t. all these variables (other than $v_1$), we get
$Q_{w_1}(v_1).$ Hence the last step of evaluating the iterated
integral is by (\ref{new2})
\begin{eqnarray*}
p_u(w)&=&\int_{v_0+v_1\leq 1}Q_{w_1}(v_1)dv_0dv_1
=\int_0^1\left[\int_0^{1-v_0}Q_{w_1}(v_1)dv_1\right]dv_0,
\end{eqnarray*}
and the result follows.
\end{proof}
Let
\begin{equation}\label{eq:pg} G_0(x)=1\ {\rm and } \ \
G_{2n}(x) =
\displaystyle\sum_{\substack{w \ {\rm Catalan,} \\
|w|=n}}Q_w(x). \end{equation}
\begin{lemma}\label{lem:cat:rec}
With
%%AB wrote G as a seq instead of just G
%$G$
$\{G_{2n}, n \geq 0\}$ as in~\eqref{eq:pg},\\
(i) for $n\geq 1$ we have
\begin{equation*}G_{2n}(x)=\displaystyle\sum_{k=1}^n
G_{2(n-k)}(x)\displaystyle\int_0^{1-x}
G_{2(k-1)}(y)dy.\end{equation*} (ii)
\begin{equation*}G_{2n}(x)=\displaystyle\frac{(1-x)(n+1-x)^{n-1}}{n!}\end{equation*}
for all $n\geq 0$.
\end{lemma}
\begin{proof}
(i) For $n=1$, the only Catalan word of length $2$ is $aa$. Hence
$G_2(x)=Q_{aa}(x)=1-x$ and the result is true for $n=1$. Let $n\geq
2$ and let $G_{2n,k}(x)$ be the sum of $Q_w(x)$ over all Catalan
words $w$ such that the first letter is repeated at the $2k$-th
place. Clearly, for such a $w$, $w=aw_1aw_2$ where $w_1$ is a
Catalan word of length $2(k-1)$ and $w_2$ is a Catalan word of
length $2(n-k)$. Using the previous lemma, it follows that
\begin{eqnarray*}
G_{2n,k}(x)&=&\sum_{|w_1|=(k-1),|w_2|=(n-k)}Q_{aw_1aw_2}(x)\\
&=&\sum_{|w_1|=(k-1),|w_2|=(n-k)}Q_{aw_1a}(x)Q_{w_2}(x)\\
&=& \sum_{|w_1|=(k-1)}Q_{aw_1a}(x)\sum_{|w_2|=(n-k)}Q_{w_2}(x)\\
&=& \sum_{|w_1|=(k-1)}\int_0^{1-x}Q_{w_1}(y)dy \sum_{|w_2|=(n-k)}Q_{w_2}(x)\\
&=& \sum_{|w_2|=(n-k)}Q_{w_2}(x) \int_0^{1-x}(\sum_{|w_1|=(k-1)}Q_{w_1}(y))dy\\
&=& G_{2(n-k)}(x)\int_0^{1-x}G_{2(k-1)}(y)dy.
\end{eqnarray*}
 As $G_{2n}(x)=\sum_{k=1}^n G_{2n,k}(x)$, part (i) follows.\\
%\end{proof}

%\begin{lemma}
%\end{lemma}

%\begin{proof}
\noindent (ii) We prove this by induction. The cases $n=0,1$ are
clear. Now
%let us
suppose that the result is true for all $j<n$.
Then by the last lemma and the induction hypothesis we have
\begin{eqnarray*}
G_{2n}(x)&=& \sum_{k=1}^n G_{2(n-k)}(x)\int_0^{1-x} G_{2(k-1)}(y)dy\\
&=& \sum_{k=1}^n\frac{(1-x)(n-k+1-x)^{n-k-1}}{(n-k)!}\int_0^{1-x} \frac{(1-y)(k-y)^{k-2}}{(k-1)!}dy\\
&=&\sum_{k=1}^n\frac{(1-x)(n-k+1-x)^{n-k-1}}{(n-k)!}\frac{(1-x)(k-1+x)^{k-1}}{k!}\\
&=& \frac{(-1)^{n-1}}{n!}\sum_{k=1}^n \binom n k p_k(z)p_{n-k}(-z),
\end{eqnarray*}
where $z=1-x$ and
\begin{equation*}p_n(x)=x(x-n)^{n-1}\end{equation*} is the Abel Polynomial
of degree $n$. It is a well known fact (see \citet{rio}) that Abel polynomials satisfy the following
combinatorial identity:
\begin{equation*}p_n(x+y)=\sum_{k=0}^n \binom n k p_k(x)p_{n-k}(y).\end{equation*}
It follows that
\begin{eqnarray*}
G_{2n}(x) &=& \frac{(-1)^{n-1}}{n!}\sum_{k=0}^n \binom n k p_k(z)p_{n-k}(-z)+\frac{(-1)^{n}}{n!}p_0(z)p_n(-z)\\
&=& \frac{(-1)^{n-1}}{n!}p_n(0)+\frac{(-1)^{n}}{n!}p_n(-z)\\
&=& \frac{(-1)^n(-z)(-z-n)^{n-1}}{n!}\\
&=& \frac{z(z+n)^{n-1}}{n!}\\
&=& \frac{(1-x)(n+1-x)^{n-1}}{n!}
\end{eqnarray*}
and the proof of Lemma~\ref{lem:cat:rec} is complete.
\end{proof}
%\begin{proof}
Now we complete the proof of Theorem~\ref{theorem:moments}. Using the previous lemmas we have
\begin{eqnarray*} \lim_{n\rightarrow
\infty}\frac{1}{n} {\rm Tr}\big(\frac{W_n^u}{\sqrt{n}}\big)^{2k}&=&
\sum_{\substack{w~{\rm Catalan,}\\ |w|=k}}p_u(w)\\
&=& \int_0^1 G_{2k}(x)dx \ \ \ \ \ {\rm (see \ equation (\ref{eq:pg}))}\\
&=& \int_0^1\frac{(1-x)(k+1-x)^{k-1}}{k!}dx \ \ {\rm ( Lemma~\ref{lem:cat:rec}(ii))}\\
&=& \frac{k^k}{(k+1)!}.
\end{eqnarray*}
\end{proof}
\section{Some comments on other variants and joint convergence}\label{sec:comments}
%\vspace{-1em}
\begin{enumerate}
\item It can be shown that the limit which appears in the triangular Wigner
matrix can appear in a variety of other matrices which satisfy the so called {\it Property P} introduced in~\cite{bana:bose}. The Property P states that
\begin{equation*}M^*=\sup_n\sup_{i,j\leq n}|\{1\leq k\leq n : k+i\leq n+1,\  k+j \leq n+1, \ L(k,i)=L(k,j)\}|<\infty.\end{equation*}
The proof of this fact is  similar to the proofs in
in~\cite{bana:bose} and hence we skip it.\\

\item The joint convergence of patterned random matrices was initiated in~\cite{bosehazrasaha} and was
further studied by~\cite{basu:bose:gang:hazra:2011a} for multiple
copies of different independent patterned random matrices. We can so
study the joint convergence of independent patterned triangular
matrices along the same lines. It can be shown that if the entries
have uniformly bounded moments and if $p_u(w)$ exists then
independent copies of patterned random matrices jointly converge
with respect to $\frac1n \Tr$ for any of the matrices considered in
this article. The behavior in the limit though, is a highly
non-trivial problem.\\

\item Let $W_{1,n}^u, W_{2,n}^u$ be two triangular  Wigner matrices having uniformly bounded moments. Let $(\{a_1, a_2\},\phi)$ denote the joint limit of
$(\{\frac{W_{1,n}^u}{\sqrt{n}},
\frac{W_{2,n}^u}{\sqrt{n}}\},\phi_n)$ where $\phi_n =\frac{1}{n}
{\rm E}[{\rm Tr}(\cdot)]$. Then $a_1$ and $a_2$ are not free. In
fact,
\begin{equation*}\phi(a_1^2)=\phi(a_2^2)=\lim_{n\rightarrow \infty}\frac{1}{n}{\rm
E}[{\rm Tr}(W_{1,n}^u/\sqrt{n})^2]=p_{u,W}(aa)=1/2.\end{equation*}
Also, from Table 1,
\begin{eqnarray*}
\phi(a_1^2a_2^2)&=&\lim_{n\rightarrow \infty}\frac{1}{n}{\rm E}[{\rm Tr}((W_{1,n}^u/\sqrt{n})^2(W_{2,n}^u/\sqrt{n})^2)]\\
&=&p_{u,W}(aabb)=1/3.
\end{eqnarray*}
As $\phi(a_1^2a_2^2)\neq \phi(a_1^2)\phi(a_2^2)$, $a_1$ and $a_2$
are not free.\\

\item Another variant of the triangular Wigner matrix is
\begin{equation*}
        W_n^l =
            \left[ \begin{array} {cccccc}
                0 & 0 & 0 & \ldots & 0 & x_{1n} \\
                0 & 0 & 0 & \ldots & x_{2(n-1)} & x_{2n} \\
                      &       &       & \vdots &       &   \\
                0 &x_{2(n-1)}& x_{3(n-1)}& \ldots  & x_{(n-1)(n-1)} & x_{(n-1)n} \\
                x_{1n} & x_{2n} & x_{3n} & \ldots & x_{(n-1)n} & x_{nn}
            \end{array} \right].
\end{equation*}
It should be noted that we can always delete the diagonal from the
above matrix without changing the limiting spectral distribution.
Let $P$ be the following orthogonal matrix
\begin{equation}\label{def:wigner}
        P =
            \left[ \begin{array} {cccccc}
                0 & 0 & 0 & \ldots & 0 & 1 \\
                0 & 0 & 0 & \ldots & 1 & 0 \\
                      &       &       & \vdots &       &   \\
                0 & 1 & 0 & \ldots  & 0 & 0 \\
                1 & 0 & 0 & \ldots & 0 & 0
            \end{array} \right].
\end{equation}
Then we have
\begin{equation}
        PW_n^uP' =
            \left[ \begin{array} {cccccc}
                0 & 0 & 0 & \ldots & 0 & x_{1n} \\
                0 & 0 & 0 & \ldots & x_{2(n-1)} & x_{1(n-1)} \\
                      &       &       & \vdots &       &   \\
                0 & x_{2(n-1)} & 0 & \ldots  & x_{22} & x_{12} \\
                x_{1n} & x_{1(n-1)} & x_{1(n-2)} & \ldots & x_{12} & x_{11}
            \end{array} \right]
\end{equation}
which can be seen to be equal to $W_n^l$ upon renaming of the
variables. Since conjugation by orthogonal matrices preserves
eigenvalues, we can conclude that  the limiting spectral
distribution of $W_n^l/\sqrt n$ is same as that of $W_n^u$.\\

\item Interestingly, $W_n^u/\sqrt n$ and $W_n^l/\sqrt n$ jointly converge.  If we call the limiting variables as
$a_1, a_2$, then again it can be shown that they are not free. But
it easily follows that $a_1+a_2$ obeys the
%there $S$ is a variable having
the semicircular law.\\

\item  As is well known, for the full version of the matrices, the contribution of every Catalan word to the limiting moment equals
one for all the common patterned matrices. Further, for the
Symmetric Circulant every word contributes one. For the Reverse
Circulant, only the so called symmetric words contribute, each
contributing one. For the Hankel matrices also, only the symmetric
words contribute but all of them do not contribute one or even
contribute equally. See for example~\cite{bosehazrasaha}.
 It is an interesting question to study the individual contribution of each word  to the limiting moment for the triangular matrices.
 This does not appear to be easy, even for the Wigner case.\\

\item The following facts on word contribution can be proved for the triangular Wigner matrices.
The proof is given in Section~\ref{section:appendix}.\vskip7pt

\noindent (i) For any word $w$ of length $2k$, $p_u(w)  \leq
1/(k+1)$.\vskip5pt

\noindent (ii) If $w$ is of the form $w= aabbcc\ldots$ then $p_u(w)
=1/(k+1)$.\vskip5pt

\noindent (iii) If $w$ and $w^\prime$ are Catalan words of the form
$w=aw_1aw_2$ and $w^\prime =aaw_1w_2$  then $p_u(w) \leq
p_u(w^\prime)$.\vskip5pt

\noindent
(iv) If $w$ and $w^\prime$ are Catalan words of the form $w = abbaw_1w_2$ and $w^\prime = abw_1 baw_2$ are Catalan words then
 $p_u (w) \geq p_u (w^\prime)$.\\

%\section{Some simulations}

\item We wish to point out some histogram plots for the patterned triangular matrices. In Figures
\ref{fig:1} and \ref{fig:2}  we have plotted the histograms of the
empirical spectral distribution of the symmetric triangular Wigner,
Toeplitz, Hankel and Symmetric Circulant matrices.

It can be easily checked from the density equation given earlier in
\ref{eq:densityequation} that the density is unbounded at zero for
the triangular Wigner. Further, from the moment formula, one can
easily check that the support of the LSD is contained in
$[-\sqrt{e}, \ \sqrt{e}]$. We can see evidence of both these facts
in the histogram. It is known that the maximimum eigenvalue of the
full Wigner matrix converges almost surely to $2$. We believe that
likewise, for the triangular Wigner, the maximum eigenvalue
converges to  $\sqrt{e}$ almost surely.

As pointed out earlier, obtaining moment properties of the LSD for
the other three matrices does not seem to be easy. However, it
appears from the simulations that all the other three LSDs also have
densities and these are also unbounded at zero. Moreover all these
LSDs have unbounded support, just like for their full matrix
versions.

Interestingly, it is known that the LSD of the full Hankel is not
unimodal (simulations results show that it is bimodal but there is
no formal proof). However, the simulation evidence implies that the
triangular Hankel LSD is unimodal. Recall that for the full Hankel
matrix, the entries in the upper triangle and the entries in the
lower triangle (ignoring the main antidiagonal part) are independent
of each other. One wonders whether this particular nature has
anything to do with the bimodality of the LSD of the full Hankel.\\

\item
Now suppose we consider triangular patterned matrices but we drop
the assumption of symmetry. Let $A_n$ be any such matrix. It can be
shown by the moment method that the LSD  of the symmetric matrix
$A_nA_n^\prime/n$ exists. A natural (and hard) question is what
happens to the LSD for the asymmetric matrix $A_n/\sqrt{n}$. The
moment method cannot be directly applied and one has to resort to
the Stieltjes transform.\\
%We are currently working on this  problem.

\item Let $X^{(s)}$ be a symmetric matrix with link function $L$ and
let $X$ be it asymmetric version as defined
in~\cite{bose:sreela:sen:2010}. A general question that was raised
there was the following. Suppose $X^{(s)}/\sqrt n$ has an LSD
identified by a random variable say $X$. Suppose $XX'/n$ has an LSD
denoted by $Y$ say. When do $X^2$ and $Y$ have the same
distribution?

The answer is affirmative  when $X^{(s)}$ is a symmetric Wigner
matrix and $X$ is its asymmetric version (the matrix with i.i.d.\
entries) from the relation between Mar\v{c}enko-Pastur law and the
Wigner law. In~\cite{bose:sreela:sen:2010} it was shown that this is
true for the Toeplitz matrix but not for Hankel or Reverse Circulant
matrices.

The result in this article gives another important example in the
form of the triangular Wigner matrix.
\end{enumerate}

\begin{figure}[htp]
\centering
\includegraphics[height=70mm, width =70mm ]{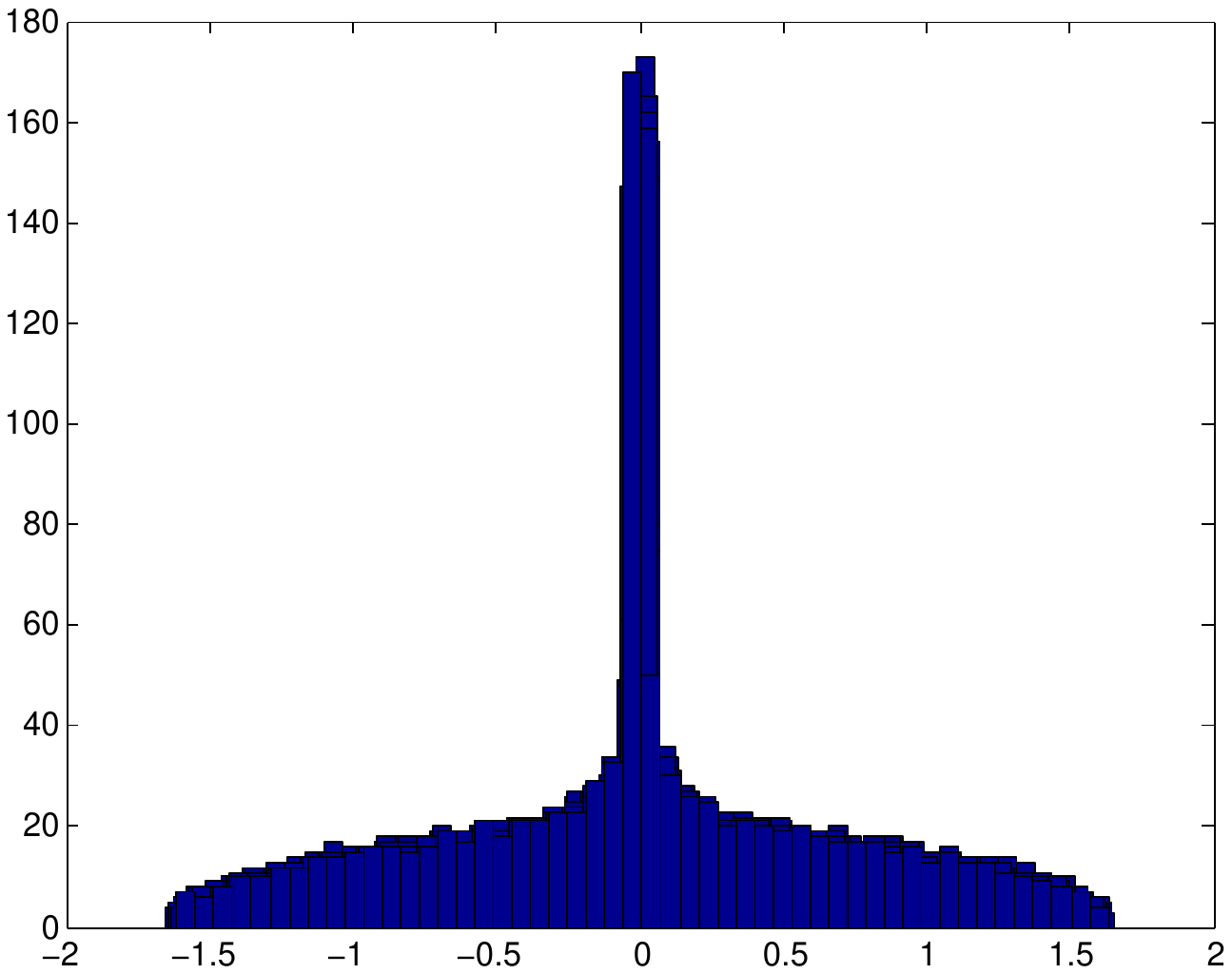}
\includegraphics[height=70mm, width =70mm ]{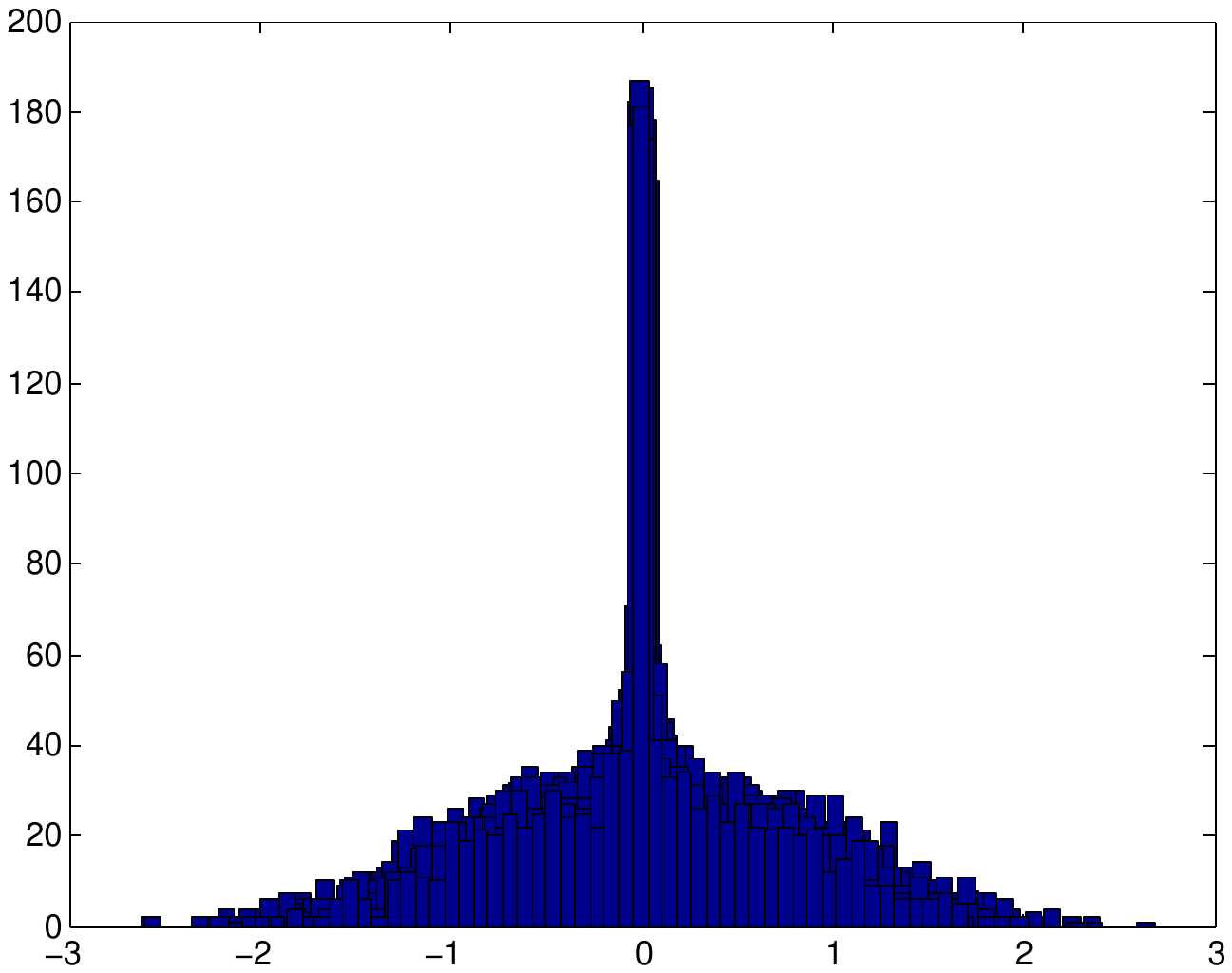}
\caption{\footnotesize  Histogram plots of empirical distribution of
triangular  matrices ($n=1000$) with entries $N(0,1)$ of (i) Wigner
(left) (ii) Toeplitz (right)}
%Histogram plot of empirical distribution of triangular  matrix($n=1000$) with $N(0,1)$ entries.}
\label{fig:1}
%\label{fig:case1&2}
\end{figure}
\begin{figure}[htp]
\centering
\includegraphics[height=70mm, width =70mm ]{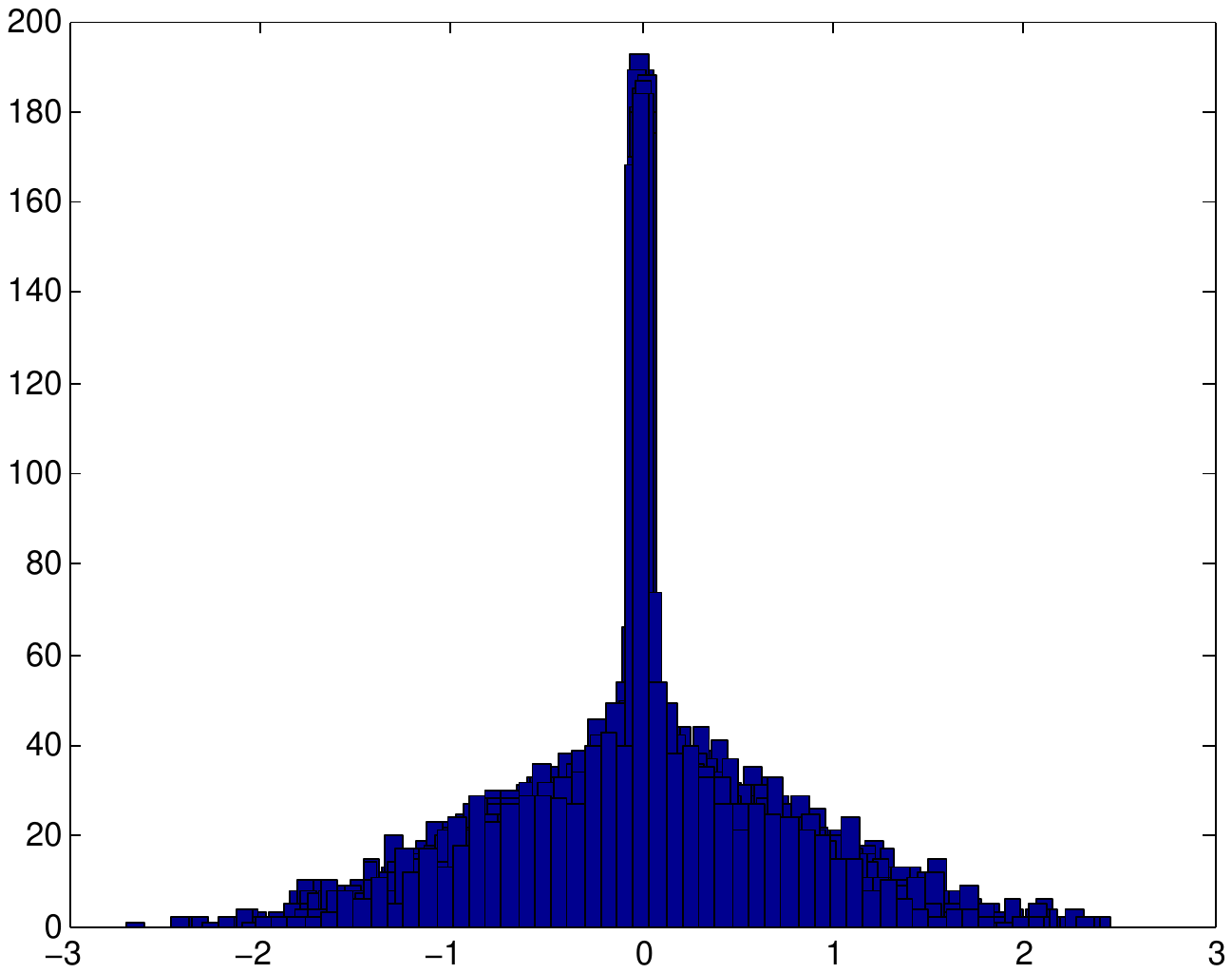}
\includegraphics[height=70mm, width =70mm ]{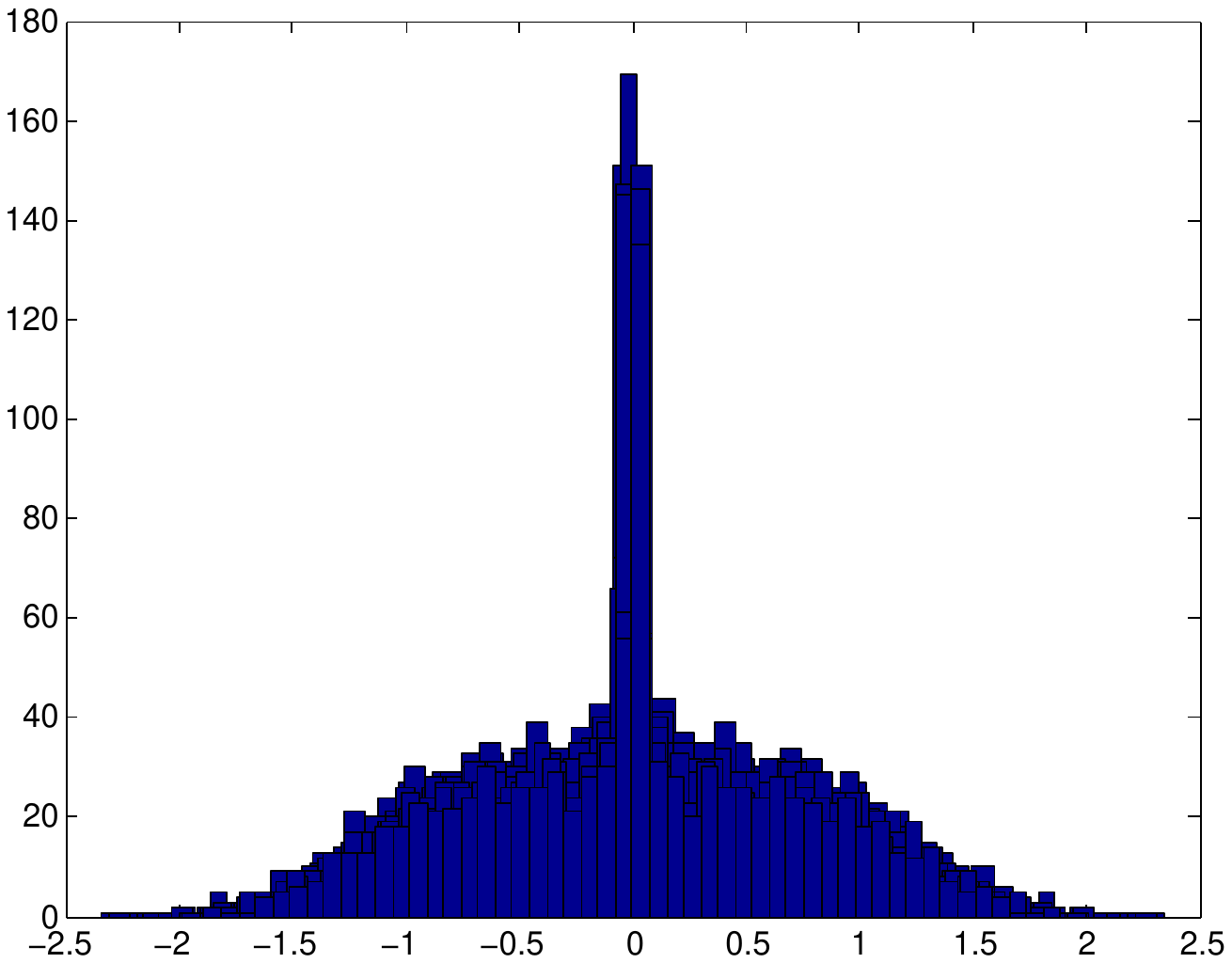}
\caption{\footnotesize  Histogram plots of empirical distribution of
triangular matrices ($n=1000$) with entries $N(0,1)$ of (i)
Symmetric Circulant (left)  (ii) Hankel (right) }%\vskip-50pt
% matrix ($n=500$) with $N(0,1)$ entries.}%\label{fig:case3&4}
\label{fig:2}
\end{figure}
\normalsize
\vskip50pt
\section{Appendix}\label{section:appendix}

\subsection{Contribution of different Catalan Words for triangular Wigner Matrix}

%Even for triangular Wigner matrix, 
It appears difficult, even for triangular Wigner matrix, to determine 
%the value of 
$p_u(w)$ for each Catalan word $w$. However, $p_u(w)$ can be determined for some special classes of words. 
As we have already seen, 
unlike what happens for a full Wigner matrix, $p_u(w)$ is not same for all Catalan words. Here we record some observations about contributions of different Catalan words for triangular Wigner matrices. This  might be useful in obtaining more information about the individual $p_u(w)$'s. We start with a simple lemma.

\begin{lemma}\label{lemma:aabbcat}
Let $w$ be the Catalan word of length $2k$, $w=a_1a_1a_2a_2...a_ka_k$. Then $p_u(w)=\frac{1}{k+1}.$
\end{lemma}

\begin{proof}
Clearly, $\pi \in \Pi^*(w)$ implies $$ (\min(\pi(0),\pi(1)), \max(\pi(0),\pi(1)))=(\min(\pi(1),\pi(2)), \max(\pi(1),\pi(2)))$$
and hence $\pi(0)=\pi(2)$. Arguing similarly, it follows that,

$$\Pi^*(w)=\{\pi: \pi(0)=\pi(2)=\pi(4)=\ldots =\pi(2k)\} .$$

Hence,

$$|\Pi_1^*(w)|=\#\{\pi: \pi (0)+\pi(1)\leq n,\pi (0)+\pi (2)\leq n,\ldots,  ,\pi(0)+\pi(k)\leq n \}. $$

Now 
%by the usual arguments of volume method, it follows that
\begin{eqnarray*}
p_u(w)&=&\lim_{n \rightarrow \infty}\frac{|\Pi_1^*(w)|}{n^{1+k}}\\
&=& \idotsint\limits_{[0,1]^{k+1}}I(v_0+v_1\leq 1, v_0+v_2\leq 2, v_0+v_3\leq 1,...,v_0+v_k \leq 1)dv_0dv_1...dv_k\\
&=& \int_0^1\left[\int_0^{1-v_0}\cdots \int_0^{1-v_0}dv_1\cdots dv_k \right]dv_0\\
&=& \int_0^1(1-v_0)^kdv_0\\
&=& \frac{1}{k+1}
\end{eqnarray*}
which completes the proof.
\end{proof}

Now we show that words in some classes contribute more than words in some other classes.
Before that we need the following two lemmas.

\begin{lemma}
\label{circuit1a}
Let $w$ be a Catalan word of length $2k$. Let $\pi$ be a  function $\pi: \{0,1,...,k\}\rightarrow \{1,2,...,n\}$ such that $w[i]=w[j] \Rightarrow (\pi(i-1),\pi(i))=(\pi(j),\pi(j-1))$. Then $\pi(0)=\pi(2k)$ and hence $\pi \in \Pi^*(w)$.
\end{lemma}
\begin{proof}
The proof is easy by induction and we omit the details.
\end{proof}

%The next lemma is also very easy and we omit the proof.

%\begin{lemma}
%\label{circuit2}
%Let $w=w_1w_2$ be a Catalan word of length $2k$ where $w_1$ and $w_2$ are Catalan words of length $2k_1$ and $2k_2$ respectively. Then for every vertex $i>2k_1$, $\phi(i)>2k_1$ or $\phi(i)=\phi(2k_1).$
%\end{lemma}

\begin{lemma}
\label{max1}
Let $w$ be a Catalan word of length $2k$ such that $w=aw_1aw_2$, where $w_1$ and $w_2$ are Catalan words. Let $w'$ denote the Catalan word $aaw_1w_2$. Then $p_u(w)\leq p_u(w').$
\end{lemma}

\begin{proof}
Let us denote, for any set $U \in \mathbb{R}^{k+1}$, by $Vol(U)$, the Lebesgue measure of the set $U \bigcap [0,1]^{k+1}$. From what we have already proved, it follows that for every Catalan word $w_0$ of length $2k$, $p_u(w_0)=Vol (U_0)$ where $U_0$ is the region in $\mathbb{R}^{k+1}$ given by the set of all $v$'s determined by the set of inequalities 
\begin{equation}
N=\{v_{\phi(i-1)}+v_{\phi(i)}\leq 1, i \in S-\{0\}\}.
\end{equation}

Now let $w=aw_1aw_2$ and let $w_1$ and $w_2$ be Catalan words of lengths $2k_1$ and $2k_2$. Now we partition the set $N$ 
%divide the inequalities 
%$v_{\phi(i-1)}+v_{\phi(i)}\leq 1$ 
into three classes.

\begin{enumerate}
\item

 $N_1=\{v_0+v_1 \leq 1\}$: This is the  first inequality since $\phi(0)=0, \phi(1)=1$ and $1\in S$.

\item

 $A_{w_1}(v_1)$: this is the set of all inequalities  $\{v_{\phi(i-1)}+v_{\phi(i)}\leq 1: 2\leq i \leq 2k_1+1, i \in S-\{0\} \}$; i.e. the inequalities corresponding to the generating vertices in $w_1$. By Lemma \ref{circuit1} it follows that these inequalities do not involve the variable $v_0$. Also, the inequalities only depend on $w_1$ and $v_1$.

\item

$A_{w_2}(v_0)$: this is the set of all the inequalities  $\{v_{\phi(i-1)}+v_{\phi(i)}\leq 1: 2k_1+3\leq i \leq 2k, i \in S-\{0\} \}$; i.e. the class of inequalities corresponding to the generating vertices in $w_2$.  From Lemma \ref{circuit1a}, it follows that $\pi(2k+2)=\pi(0)$ for every $\pi \in \Pi^*(w)$ and hence $\phi(2k+2)=0$. Now Lemma \ref{circuit1} implies that all the variables occurring in these inequalities are independent of the variables that occurred previously, apart from $v_0$. Also, the inequalities only depend on $w_1$ and $v_0$.

\end{enumerate}

It follows that $p_u(w)=Vol(B_w)$ where

$$B_w=N_1 \cup A_{w_1}(v_1)\cup A_{w_2}(v_0).$$
%\{v_0+v_1\leq 1, A_{w_1}(v_1)\leq 1, A_{w_2}(v_0)\leq 1\}.$$
Similarly, for the word $w'=aaw_1w_2$ the three classes of inequalities are:

\begin{enumerate}

\item

 $N_1$: 
%$That this is the first inequality follows from noting that $\phi(0)=0, \phi(1)=1$ and $1\in S$.

\item

 $A_{w_1}(v_0)$: in this case $\phi(2)=0$. By renaming the variables other than $v_0$, if necessary,
the class of inequalities corresponding to the generating vertices in $w_1$ in this case is same as that in the previous case with $v_0$ replaced by $v_1$. By Lemma \ref{circuit1} it follows that these inequalities do not involve the variable $v_1$.
%
% the set  of  inequalities is  $\{v_{\phi(i-1)}+v_{\phi(i)}\leq 1: 2\leq i \leq 2k_1+1, i \in S-\{0\} \}$; i.e. 
\item

$A_{w_2}(v_0)$:  the class of inequalities corresponding to the generating vertices in $w_2$ in this case is same as that in the previous case after renaming the variables other than $v_0$. Lemma \ref{circuit1} implies that all the variables occurring in these inequalities are independent of the variables occurred previously, apart from $v_0$.

\end{enumerate}

It follows that $p_u(w)=Vol(B_{w'})$ where

$$B_{w'}=N_1\cup A_{w_1}(v_0)\cup A_{w_2}(v_0).$$

%\{v_0+v_1\leq 1, A_{w_1}(v_0)\leq 1, A_{w_2}(v_0)\leq 1\}.$$

Now
\begin{eqnarray*}Vol(B_w)&=&Vol(B_w^1)+ Vol(B_w^2) ~~\text{and} \\
Vol(B_{w'})&=&Vol(B_{w'}^1)+ Vol(B_{w'}^2)
\end{eqnarray*}
where

\begin{eqnarray*}
B_w^1&=&N_1\cup A_{w_1}(v_1)\cup A_{w_2}(v_0)\cup \{v_0\leq v_1\},\\
B_w^2&=&N_1\cup A_{w_1}(v_1)\cup A_{w_2}(v_0)\cup \{v_1\leq v_0 \}
\end{eqnarray*}

and
\begin{eqnarray*}
B_{w'}^1&=&N_1\cup A_{w_1}(v_0)\cup  A_{w_2}(v_0)\cup  \{v_0\leq v_1\},\\
B_{w'}^2&=&N_1 \cup  A_{w_1}(v_0)\cup  A_{w_2}(v_0)\cup \{v_1\leq v_0 \}.
\end{eqnarray*}
%As every equation in the family $A_{w_1}$ is of the form $v_i+v_j\leq 1$, $v_0\leq v_1$ and $A_{w_1}(v_1)\leq 1$ together imply $A_{w_1}(v_0)\leq 1.$ 
It is easy to see that $B_{w'}^1\supseteq B_w^1$ and 
%. Similarly it follows that 
$B_{w}^2\supseteq B_{w'}^2.$

Now,
\begin{eqnarray*}
B_{w'}^1- B_w^1&=&N_1\cup A_{w_1}(v_0)\cup A_{w_2}(v_0)\cup \{ v_0\leq v_1\}\cup (A_{w_1}(v_1))^\prime; \\
B_{w}^2- B_{w'}^2&=&n_1 \cup A_{w_1}(v_1)\cup A_{w_2}(v_0)\cup \{ v_1\leq v_0\}\cup (A_{w_1}(v_0)^\prime. 
\end{eqnarray*}
%Here by $A_{w_1}(v_1)\nleq 1$ we mean that some of the inequalities $A_{w_1}(v_1)$ is not satisfied and $A_{w_1}(v_0)\nleq 1$ is interpreted similarly. 
Now by interchange of variables $v_0$ and $v_1$, it follows that $Vol(B_{w'}^1- B_w^1)=Vol(C_{w,w'})$ where

$$C_{w,w'}= N_1 \cup A_{w_1}(v_1)\cup A_{w_2}(v_1)\cup \{v_1\leq v_0\}\cup (A_{w_1}(v_0))^\prime.$$

Again, $v_1 \leq v_0$ and $A_{w_2}(v_0)$ together imply $A_{w_2}(v_1)\leq 1$ and hence $C_{w,w'}\supseteq B_{w}^2- B_{w'}^2.$ It follows that,
\begin{eqnarray*}
Vol (C_{w,w'})&\geq & Vol (B_{w}^2- B_{w'}^2)\\
\Rightarrow Vol(B_{w'}^1- B_w^1) &\geq & Vol (B_{w}^2- B_{w'}^2)\\
\Rightarrow Vol(B_{w'}^1)+Vol(B_{w'}^1) &\geq & Vol (B_{w}^1)+Vol B_{w}^2)
\end{eqnarray*}

and this completes the proof.

\end{proof}

The next Lemma is similar to the previous one.

\begin{lemma}
Let $w=abbaw_1w_2$ and $w'=abw_1baw_2$ be two Catalan words where $w_1$ and $w_2$ are Catalan words. Then $p_u(w)\geq p_u(w').$
\end{lemma}

\begin{proof}
We use the same notations. As in the previous Lemma, 
now $p_u(w)=Vol(B_w)$ and $p_u(w')=Vol(B_{w'})$ where

\begin{eqnarray*}
B_w&=&N_1\cup \{ v_1+v_2\leq 1\}\cup A_{w_1}(v_0)\cup A_{w_2}(v_0) ~\text{and} \\
B_{w'}&=&N_1\cup\{ v_1+v_2\}\cup A_{w_1}(v_2)\cup A_{w_2}(v_0).
\end{eqnarray*}

Now
\begin{eqnarray*}
Vol(B_w)&=&Vol(B_w^1)+ Vol(B_w^2) ~~\text{and}\\
Vol(B_{w'})&=&Vol(B_{w'}^1)+ Vol(B_{w'}^2)
\end{eqnarray*}

where
\begin{eqnarray*}
B_w^1 &=&N_1\cup \{ v_1+v_2\leq 1\}\cup A_{w_1}(v_0)\cup A_{w_2}(v_0)\cup \{ v_0\leq v_2\},\\
B_w^2&=&N_1\cup \{ v_1+v_2\leq 1\}\cup A_{w_1}(v_0)\cup A_{w_2}(v_0)\cup \{ v_2\leq v_0\},\
%\{v_0+v_1\leq 1, v_1+v_2 \leq 1, A_{w_1}(v_0)\leq 1, A_{w_2}(v_0)\leq 1, v_2\leq v_0 \}
\end{eqnarray*}

and
\begin{eqnarray*}
B_{w'}^1 &=&
N_1\cup \{ v_1+v_2\leq 1\}\cup A_{w_1}(v_2)\cup A_{w_2}(v_0)\cup \{ v_0\leq v_2\},\\
B_{w'}^2 &=&
N_1\cup \{ v_1+v_2\leq 1\}\cup A_{w_1}(v_2)\cup A_{w_2}(v_0)\cup \{ v_2\leq v_0\}.
\end{eqnarray*}

As $v_0\leq v_2$ and $A_{w_1}(v_2)$ together imply $A_{w_1}(v_0)$, $B_{w}^1\supseteq B_{w'}^1$. Similarly it follows that $B_{w'}^2\supseteq B_{w}^2.$

Now,
\begin{eqnarray*}
B_{w}^1- B_{w'}^1&=&N_1 \cup \{v_1+v_2\leq 1\} \cup A_{w_1}(v_0)\cup A_{w_2}(v_0) \{ v_0\leq v_2\} \cup 
(A_{w_1}(v_2))^\prime;\\
B_{w'}^2- B_{w}^2&=&
N_1 \cup \{v_1+v_2\leq 1\} \cup A_{w_1}(v_2)\cup A_{w_2}(v_0) \{ v_2\leq v_0\} \cup 
(A_{w_1}(v_0))^\prime .
\end{eqnarray*}

Once again by interchange of two variables $v_0$ and $v_2$ we see that $Vol(B_{w'}^2- B_w^2)=Vol(C_w,w')$ where

$$C_{w,w'}= N_1  \cup \{v_1+v_2\leq 1\} \cup A_{w_1}(v_0)\cup A_{w_2}(v_2) \{ v_0\leq v_2\} \cup  
(A_{w_1}(v_2)^\prime.$$

Again, $v_0 \leq v_2$ and $A_{w_2}(v_2)$ together imply $A_{w_2}(v_0)$ and hence $C{w,w'}\subseteq B_{w}^1- B_{w'}^1.$ It follows that,
\begin{eqnarray*}
Vol (C{w,w'})&\leq & Vol (B_{w}^1- B_{w'}^1)\\
\Rightarrow Vol(B_{w'}^1- B_w^1) &\geq & Vol (B_{w}^2- B_{w'}^2)\\
\end{eqnarray*}
and the proof is completed as in the previous Lemma.
\end{proof}

The next proposition gives an upper bound on $p_u(w)$'s.

\begin{proposition}
Let $w$ be a Catalan word of length $2k$. Then $p_u(w)\leq \frac{1}{k+1}.$
\end{proposition}

\begin{proof}
If $w=a_1a_1a_2a_2...a_ka_k$, then the result follows from Lemma \ref{lemma:aabbcat}. If not, by left rotation, we can obtain, from $w$, a word $w'$ such that $w'$ does not start with a double letter and $w'$ is not of the form $aw_1a$ either. Since $p_u(w)$ is invariant under rotation, $p_u(w)=p_u(w')$. By hypothesis, $w'$ must be of the form $w'=aw_1aw_2$ where $w_1$ and $w_2$ are Catalan words. By Lemma \ref{max1}, $p_u(w)=p_u(w')\leq p_u(aaw_1w_2)=p_u(w_1w_2aa)$ and the number of consecutive double letters at the end of $w_1w_2aa$ is strictly greater than that of $w$. The proof can now be completed by induction on the number of consecutive double letters at the end of $w$.
\end{proof}

We can see from the above results that $p_u(w)$ depends upon the Noncrossing structure of the word $w$. It is an interesting combinatorial problem to obtain a formula for $p_u(w)$ as a function of the Catalan structure of $w$ or to investigate what Catalan structures gives rise to the same $p_u(w)$'s.

\end{document}